\documentclass[10pt,reqno]{amsart}

\usepackage[utf8]{inputenc}
        \usepackage[T1]{fontenc}
        \usepackage{lmodern}
        \usepackage{graphicx}
        \usepackage{caption}
         \usepackage{floatrow}

\usepackage{amsmath,amsthm,amscd,amssymb,bbm,graphicx, color}

\usepackage{mathrsfs} %gives fancy script A font for attractor
\usepackage{geometry}          
\usepackage{graphics,multicol}%changepage}
\usepackage{epstopdf}
\usepackage{enumerate}

\usepackage{dsfont}

\newtheorem{theorem}{Theorem}

\newtheorem{lemma}[theorem]{Lemma}
\newtheorem{proposition}[theorem]{Proposition}

\title[The frog model on homogeneous trees with geometric lifetime]{Critical parameter of the frog model on homogeneous trees with geometric lifetime}
\date{\today}
 
\author{Sandro Gallo}
\address{Departamento de Estatística, Universidade Federal de São Carlos}
\email{sandro.gallo@ufscar.br}
\author{Caio Pena}
\address{Instituto Federal de São Paulo, Campus Araraquara}
\email{caio.pena@ifsp.edu.br}
\thanks{SG was supported by FAPESP (Auxílio Regular: 2019/23439-4) and CP thanks financial support of IFSP during the realization of his PhD}

\begin{document}
\maketitle 
 \begin{abstract}
We consider the frog model with geometric lifetime (parameter $1-p$) on homogeneous trees of dimension $d$. In 2002, \cite{alves2002-2} proved that there exists a critical lifetime parameter $p_c\in(0,1)$ above which infinitely many frogs are activated with positive probability, and they gave lower and upper bounds for $p_c$.  Since then, the literature on this model focussed on refinements of the upper bound. In the present paper we improve the bounds for $p_c$ \emph{on both sides}. We also provide a discussion comparing the bounds of the literature and their proofs. Our proofs are based on coupling.\\


\end{abstract}
\section{Introduction}

Frog models are simple models for propagation of information (or rumor, or disease) through a graph: active (informed/ill) particles perform independent random walks on the graph, activating (informing/infecting) frogs of the visited vertices. This class of models seems to have been first introduced by \cite{telcs1999branching} under the name of \emph{egg model}: particles  making independent simple symmetric random walks on $\mathbb Z^d$ transform eggs of the visited vertices into active particles. In our terminology, particles are frogs and eggs (inactive particles) are inactive frogs. 

The literature on frog models has grown very fast over the last two decades, with variations depending on the lifetime of the frogs, the underlying graph structure on which the frogs wander around, the type of random walk the frogs perform, the initial number of frogs per vertex \emph{etc...} 
In the present paper we stand in the line of \cite{alves2002-2} who considered the case where  frogs have geometric lifetime and make simple symmetric random walks, and we focus specifically on homogeneous infinite trees. %Let us mention, to conclude, that, since the seminal paper \cite{alves2002-2}, the literature on frog models has grown very fast (see for instance, just in the last 3 years, \cite{dobler2018recurrence}, \cite{hermon2018frogs}, \cite{hoffman2019infection}, \cite{kubota2019deviation}, \cite{johnson2019sensitivity}, \cite{nakajima2019first}, \cite{muller2020transience}, \cite{deijfen2020initial} between many others). This recent literature considers variants of the model with different assumptions on the lifetime, different graphs ($\mathbb Z^d$, random graphs, (ir)regular trees), different types of random walk performed by the frogs (with drift or not), different initial distributions of inactive frogs on the graph (deterministic or random)... Since most of these works assume infinite lifetime, survival is certain and different questions arise, such as  recurrence \emph{vs.} transience, shape theorem, first passage time, importance of the  initial configuration...

More precisely, consider the homogeneous tree $\mathbb T_d$, that is, the rooted tree in which each vertex has (is connected by edges to) $d+1$ neighbours. One frog is put at each vertex and all but the one of the root start inactive. Active frogs perform  simple symmetric random walks on $\mathbb T_d$, for a geometric (parameter $1-p$) number of steps\footnote{The geometric starts at $0$, that is, the probability to make $n\ge0$ steps is $(1-p)p^n$.}, activating the inactive frogs of the visited vertices. After its geometric number of steps, the active frog ``dies'': it remains inactive forever. The process \emph{survives} if infinitely many frogs are activated. 

For any $p\in[0,1]$ we denote the law of the process by $\mathbb{P}_p$. 
Naturally, if $p=0$ then the frog of the root dies and $\mathbb P_p(\text{survival})=0$, while on the other hand, if $p=1$, frogs won't die and $\mathbb P_p(\text{survival})=1$. Moreover, it is clear that $\mathbb P_p(\text{survival})$ is non-decreasing in $p$, and we can define the critical parameter for the model on $\mathbb T_d$ as
\[
p_c=p_c(d):=\inf\{p\in[0,1]:\mathbb P_p(\text{survival})>0\}. 
\]

In the present paper we are interested in obtaining tight bounds for $p_c$, as functions of $d$, so let us list rapidly the results of the literature concerning this question.  In their seminal paper in 2002,  \cite{alves2002-2} proved that 
\begin{equation}\label{eq:bounds_seminal}
\frac{d+1}{2d+1}\le p_c\le \frac{d+1}{2d-2}\,.
\end{equation}
Three years later, \cite{lebensztayn/machado/popov/2005} was dedicated to improve the upper bound to
\[
p_c\le \frac{d+1}{2d}.
\]
It is only in 2018, with the paper of \cite{gallo2018frog}, that a new improvement of the upper bound was obtained as a consequence of the study of a percolation model on \emph{oriented} trees:
\[
p_c\le \frac{(d+1)[(7d-1)-\sqrt{(7d-1)^2-14}] }{d(7d-1)^2-7d+2-d(7d-1)\sqrt{(7d-1)^2-14}}\,.
\]
One year after,  \cite{lebensztayn2019new} dedicated to one further improvement of the upper bound, but the obtained expression is too complicated to be stated here (see Definition 2.1 and Theorem 2.2 therein).

Notice that the above improvements all refer to the upper bound of $p_c$. The  lower bound in display \eqref{eq:bounds_seminal}, obtained by \cite{alves2002-2}, was not improved so far. 

In this paper we prove the following result. 
 
\begin{theorem}\label{theorem:bounds}
For any $d\ge2$,
\begin{equation}\label{eq:bounds}
         \frac{2 (d+1)}{\sqrt{4 d^2+4 d-3}+2 d+1}\leq p_c(d) \leq \frac{(d+1)({2-\frac{1}{14d^{2}}-4d})(5d-8d^{2})}{(5d-8d^{2})^2+d\left({2-\frac{1}{14d^{2}}-4d}\right)^{2}}\,.
    \end{equation}
\end{theorem}

\vspace{0.1cm}

So, since the definition of the model by \cite{alves2002-2}, two papers were dedicated exclusively to obtain improvements on the upper bound (\cite{lebensztayn/machado/popov/2005} gave a simple bound and \cite{lebensztayn2019new} gave an improvement, however complicated) while another one \cite{gallo2018frog} presented, as a consequence of other results, a bound ``in the middle''. Here we get better bounds \emph{on both sides}: we provide the first improvement of the lower bound since \cite{alves2002-2}, and get a tighter and simpler upper bound than \cite{lebensztayn2019new}. There is another motivation behind this seek for tighter bounds, that we now explain. \cite{fontes/machado/sarkar/2004} proved that the critical parameter of the frog model, in general graphs, is not always a monotonic function of the graph. For homogeneous trees, the question of whether or not $p_c(d)$ is monotonically decreasing in $d$ remains open. Denote by $l(d)$ (\emph{resp.} $r(d)$) the lower (\emph{resp.} upper) bound of \eqref{eq:bounds}. What we can show (using Mathematica) is that $l(d)>r(ad)$ for any $a\ge1.75$. This implies that  $p_c(d)>p_c(ad)$ for those values of $a$, meaning for instance that $p_c(2)>p_c(4)>p_c(7)>p_c(13)...$. A proof of monotonicity using this argument needs the bounds to be enhanced, and we believe the present paper also gives new insights in this direction. 

Further related literature will be provided in Section \ref{sec:discussion}. We will also take the opportunity to discuss further the monotonicity issue, and clarify the relation between the  upper bounds of the literature and their respective proofs. As we will explain there, no significant improvement on the upper bound can be achieved using our method of proof, meaning that a different approach has to be developed to prove monotonicity by comparisons of lower and upper bounds. 
 
%To conclude this section, remark that, for large $d$'s, frogs are  unlikely to go back in direction to the root, and similarly, two active frogs on the same vertex are unlikely to jump the same vertex. Thus each frog, while active, usually activates one new frog at each step, until it dies. For this reason the model behaves as a Galton-Watson process for large $d$, and it is natural that $p_c(d)\rightarrow\frac{1}{2}$. This is actually what happens for all the bounds which were obtained above, and therefore, they are asymptotically correct.  

 The paper is organized as follows: we prove the lower bound in Section \ref{sec:lower}, we prove the upper bound in Section \ref{sec:upper}, and we conclude with a discussion in Section \ref{sec:discussion}.

\section{Proof of the lower bound}\label{sec:lower}

The idea of the proof is to look at the frog model as a branching process, and to couple it with a  two-type branching process (TTBP) in such a way that  the latter dominates the former. This TTBP is defined in Subsection  \ref{def:GW}. In Subsection \ref{def:frog_as_GW}, we give the alternative construction of the frog model  as a branching process, which we call \emph{frog model branching process} (FMBP). We prove the lower bound of Theorem \ref{theorem:bounds} in Subsection \ref{def:coupling}, by constructing the coupling between the TTBP and the FMBP.  

\subsection{A two-type branching process}\label{def:GW}
Consider  a two-type branching process in which, at each time step, exactly one particle is chosen, dies, and gives birth to a random number of individuals of each type, independently of everything else. For $a=1,2$, we denote by $p_a(i,j)$ the probability that a particle of  type $a$ generates $i$ particles of type 1 and $j$ particles of type 2:
%\begin{itemize}
%    \item types 1 particles have probability $p_1(i,j)$ to generates $i$ particles of type 1 and $j$ particles of type 2 with
\begin{align}\label{eq:type1TT}
p_1(0,0)=1-p,\,\,p_1(1,0)&=0,\,\,
p_1(2,0)=\frac{pd}{d+1},\,\,
p_1(0,1)=\frac{p}{d+1}.
\end{align}  
%\item  types 2 particles have probability $p_2(i,j)$ to generates $i$ particles of type 1 and $j$ particles of type 2 with
\begin{align}
p_2(0,0)=1-p,\,\,p_2 (1,0)=\dfrac{p}{d+1},\,\,
p_2(2,0)=\dfrac{p(d-1)}{d+1},\,\,
p_2(0,1)=\dfrac{p}{d+1}.
\end{align}
%\end{itemize}

With regard to the question of survival, this TTBP behaves as the  two-type Galton-Watson process with the same offspring distribution given by $p_1$ and $p_2$ as above. And it is well-known that a multi-types Galton-Watson process has  probability zero to survive if, and only if, the largest eigenvalue of first moment matrix is smaller or equal to $1$ (see \cite{athreya2004branching} for instance). Simple calculations show that this matrix is
\begin{equation*}
M=      \left(
\begin{array}{cc}
\frac{2 d p}{d+1} & \frac{p}{d+1} \\
\frac{(2 d-1) p}{d+1} & \frac{p}{d+1} 
\end{array}
\right)
\end{equation*}
and has largest eigenvalue 
\begin{equation*}
\rho_M=\frac{\left(\sqrt{4 d^2+4 d-3}+2 d+1\right) p}{2 (d+1)}.
\end{equation*}

In other words, if 
\[
p\le \frac{2 (d+1)}{\sqrt{4 d^2+4 d-3}+2 d+1},
\]
then the TTBP defined above will generate finitely many individuals with probability $1$.

\subsection{The frog model as a branching process}\label{def:frog_as_GW}

To compare to the above TTBP, we define the following modification in the dynamics of the frog model, which does not alter its survival probability.

\begin{enumerate}
        \item\label{item: follow one particle at time} We consider the frog model in a way that frogs move one at a time, and at each time, the frog which is chosen to make the step is arbitrary. That is, the frog of the origin makes a move (with probability $p$), activating the sleeping frog of the visited site. Then, we choose whichever frog between the activated ones, to make a move (move which is done with probability $p$ too) while the others stay frozen (don't move) and active, and can be chosen to make a move in a future time step.  This procedure slows down the process in the sense that it propagates slower on the tree, but since the random walks of each activated frog are independent, this doesn't change anything in terms of survival.  
        \item For further comparison with a two-type branching  process, we will interpret ``moves'' differently. At each step, instead of saying that the chosen frog moves to a neighboring site, we will think that it dies out, and with probability $p$ gives birth to frogs at one neighboring site: one frog if the chosen neighboring site has been already visited, and two frogs otherwise.
        \item For any $t\ge0$, we denote by $\mathcal T_t\subset \mathbb{T}_d$ the set of visited sites at time $t$.  Notice that $\mathcal T_t\subseteq\mathcal T_{t+1},t\ge0$. 
        \item {We consider the model starting from the random time $K=\inf\{t\ge0:|\mathcal T_t|=d+3\}$ where $|\cdot|$ denotes the cardinal of the set. This guarantees that at least one frog at distance $2$ of the root has been activated (see next point). Notice  that $K=\infty$ has a positive probability of occurring. } 
    \item At each time step $t$, frogs which are at the tip of $\mathcal T_t$ (meaning that they have $d$ unvisited neighboring sites) are classified as Type $1$, and the other frogs are classified as Type 2. Observe that, since we consider times $t\ge K$, at least one neighboring site of a Type 2 frog is surrounded by at least two already visited sites. This means that any activated frog in the system can be classified as either type 1 or type 2. 
\end{enumerate}

Notice that, if a frog, when it births, is of type 1, depending on the position of its vertex $v$ with respect to the evolving set $\mathcal T_t,t\ge1$, it may transform into type 2 (notice the difference in the preceding sentence between ``transform'' and  ``give birth to''). However, the inverse cannot occur, a frog which births of type 2 cannot, with time passing, transform into a type 1 frog. 

Fix $t\ge K$ and $\mathcal T_t=T$. Denote by $p_v(i,j|T)$ the offspring distribution of the frog located at $v$ inside $T$ which has been chosen to make a move, where $i$ and $j$ are respectively the numbers of offsprings of types 1 and 2. For any $v$, the distribution $p_v$ lives on $\{(0,0),(1,0),(0,1),(2,0)\}$. The location of $v$ inside $\mathcal T_t=T$ specifies the type 1 or 2 of the frog:
\begin{itemize}
\item Suppose it has type 1. Then, independently of $v$ and $T$
\begin{align}\label{eq:type1Frog}
p_v(0,0|T)=1-p\,,\,\,
p_v(1,0|T)=0\,,\,\,
p_v(2,0|T)=\frac{pd}{d+1}\,,\,\,
p_v(0,1|T)=\frac{p}{d+1}.
\end{align}
\item Suppose it has type 2. Then, there exist two integers (which depend on the location of $v$ in $T$) $a,b$, with $a\ge1$ {and $a+b\ge2$}, such that 
\begin{align}\label{middle}
 p_v(0,0|T)=1-p\,,\,\,\,p_v(1,0|T)=\frac{pb}{d+1}\,,\,\,\,p_v(2,0|T)=p-p\frac{a+b}{d+1}\,,\,\,\, p_v(0,1|T)=\frac{pa}{d+1}.
\end{align}
\end{itemize}

\subsection{Proof of the lower bound}\label{def:coupling}

We are now ready to prove our lower bound. 

\begin{proof}[Proof of the lower bound of Theorem \ref{theorem:bounds}]
Let $N_t^{FM,i},i=1,2$ (\emph{resp.} $N_t^{TT,i},i=1,2$) count the number of active particles of type $i$ in the system at time $t$ in the FMBP (\emph{resp.} in the TTBP).

Since we are only interested in bounding the critical parameter of the frog model, instead of starting from $\mathcal T_0=\{o\}$ we can start the frog model from any of the configurations satisfying $|\mathcal T_0|=d+3$. So let us start from any vectors $(N_0^{FM,1},N_0^{FM,2})$ having positive probability to be produced in the FM with $|\mathcal T_0|=d+3$, and put $(N_0^{TT,1},N_0^{TT,2})=(N_0^{FM,1},N_0^{FM,2})$. If we can couple these processes in such a way that, for any $t\ge0$, $N_t^{TT,1}\ge N_t^{FM,1}$ and $N_t^{TT,1}+N_t^{TT,2}\ge N_t^{FM,1}+N_t^{FM,2}$, then, in particular, the total number of particles in the TTBP is at least as large as the total number of activated frogs in the FMBP, at each time step.  Together with what we said in Subsection \ref{def:GW}, if
\begin{equation*}
p\le  \frac{2 (d+1)}{\sqrt{4 d^2+4 d-3}+2 d+1},
\end{equation*}
then the FMBP would not survive with probability $1$, which would conclude the proof of the theorem. 

So it only remains to prove that we can couple these processes in such a way that, for any $t\ge0$, $N_t^{TT,1}\ge N_t^{FM,1}$ and $N_t^{TT,1}+N_t^{TT,2}\ge N_t^{FM,1}+N_t^{FM,2}$. The inequalities are satisfied at $t=0$ by definition. We assume the inequalities are satisfied at time $t$, and we now want to prove that they are still satisfied at time $t+1$. To couple the processes at $t+1$, we need to couple the probability distributions $p_1,p_2$ and $p_v(\cdot|T)$. We do this using a random variable $U_{t+1}$ uniformly distributed in $[0,1]$ (independent of everything) and several partitions of $[0,1]$. First, partitions  $\mathcal P_1,\mathcal P_2$ to construct $p_1,p_2$:
\begin{align*}
\mathcal P_1&=\{I_1^{0,1},I_1^{2,0},I_1^{0,0}\}\\
\mathcal P_2&=\{I_2^{0,1},I_2^{1,0},I_2^{2,0},I_2^{0,0}\}
\end{align*}
where 
\[
I_1^{0,1}=\left[0,\frac{p}{d+1}\right[\,\,,\,\,\,I_1^{2,0}=\left[\frac{p}{d+1},p\right[\,\,,\,\,\,I^{0,0}=[p,1]
\]
and
\[
I_2^{0,1}=\left[0,\frac{p}{d+1}\right[\,\,,\,\,\,I_2^{1,0}=\left[\frac{p}{d+1},\frac{2p}{d+1}\right[\,\,,\,\,\,I_2^{2,0}=\left[\frac{2p}{d+1},p\right[\,\,,\,\,\,I_2^{0,0}=[p,1]\,.
\]
We refer to Figure \ref{fig:partitions} for a pictorial representation of these partitions.

Observe that 
\[
P(U_{t+1}\in I_1^{i,j})=p_1(i,j)\,\,\,\,\text{and}\,\,\,\,P(U_{t+1}\in I_2^{i,j})=p_2(i,j). 
\]

Now, for moving frogs of type 1 in the FMBP, no matter what is the pair $(v,T)$, we can use the partition $\mathcal P_1$ since for such frogs $p_v(\cdot|T)=p_1$. 

According to \eqref{middle}, for moving frogs of type 2 we use the partition ({recall that $a\ge1$ and $a+b\ge2$})
\[
\mathcal P_{(a,b)}=\{I_{(a,b)}^{0,1},I_{(a,b)}^{1,0},I_{(a,b)}^{2,0},I_{(a,b)}^{0,0}\}
\]
where
\[
I_{(a,b)}^{0,1}=\left[0,\frac{ap}{d+1}\right[\,\,,\,\,\,I_{(a,b)}^{1,0}=\left[\frac{ap}{d+1},\frac{(a+b)p}{d+1}\right[\,\,,\,\,\,I_{(a,b)}^{2,0}=\left[\frac{(a+b)p}{d+1},p\right[\,\,,\,\,\,I_{(a,b)}^{0,0}=[p,1]\,.
\]
Observe that, for any  $a,b$ and $i,j$, we have
\[
P(U_{t+1}\in I_{(a,b)}^{i,j})=p_v(i,j|T)\,.
\]

\begin{figure}[h]
\centering
\includegraphics[scale = 0.65]{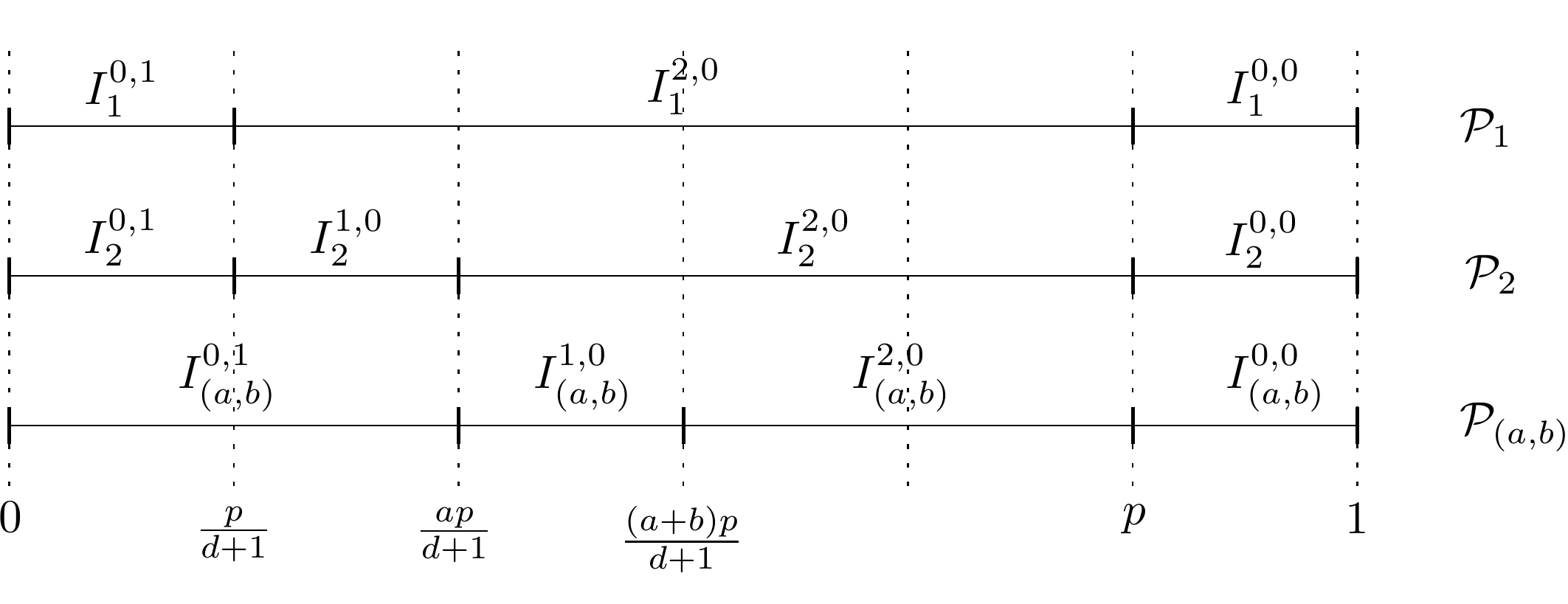}
\caption{Pictorial representation of the partitions $\mathcal P_1,\mathcal P_2$ and $\mathcal P_{(a,b)}$, when $d=4,a=2$ and $b=1$.}
\label{fig:partitions}
\end{figure}

The coupling is now performed updating both processes at time $t+1$ using the same uniform $U_{t+1}$. We can now establish the recursion from $t$ to $t+1$, recalling that the recursion hypothesis is
$N_t^{TT,1}\ge N_t^{FM,1}$ and $N_t^{TT,1}+N_t^{TT,2}\ge N_t^{FM,1}+N_t^{FM,2}$:\\

\begin{enumerate}
\item If the chosen frog for the next move of the FMBP is of type 1, then a particle of type 1 of the TTBP is chosen to give birth to its offspring. Since $N_t^{TT,1}\ge N_t^{FM,1}$, this choice is always possible. In this case, according to our coupling, the two chosen particles (one in each process) give birth the the same offspring, and thus the inequalities are maintained.\\

\item Suppose now that the chosen frog  is of type 2, then, either a type 2 particle of the branching process is chosen to give birth to its offspring, or, if there is none, a type 1 particle is chosen.\\

According to the location $v$ in $T$ of the chosen frog of type 2, a pair $(a,b)$ of integers is associated, as explained before. For any fixed pair $(a,b)$,  we have the following possibilities: \\

\begin{enumerate}
\item The particle of the TTBP is also a type 2 particle: then
\begin{itemize}
\item if $U_{t+1}\le \frac{p}{d+1}$, then one particle of type 2 is created in each process so that the inequalities are maintained;
\item if $U_{t+1}\ge p$, then the chosen particle dies in each process so that the inequalities are maintained;
\item if $U_{t+1}\in \left[\frac{p}{d+1},\frac{ap}{d+1}\right]$ (and therefore $a\ge2$ since otherwise this interval half open to the right is empty) the branching process produces two type 1 particles while the frog model produces one type 2 particle, and therefore the inequalities are maintained
%\item if $a\ge3$ then the branching process produces either one type 2 particle or one or two type 1 particles. (In the case of Figure \ref{fig:partitions}, for instance, we would have one particle of type 1 in the branching process and one particle of type 2 in the frog model). In any case, the desired inequalities are maintained;\\
%CONCLUDE!!!!
%\end{itemize}
\item if $U_{t+1}\in I_{(a,b)}^{1,0}$ then either one (if $a=1$) or two (if $a\ge2$) type 1 particles are created in the branching process, while the frog model creates one type 1 particle, thus the inequalities are maintained;
\item if $U_{t+1}\in I_{(a,b)}^{2,0}$ then two type 1 particles are created in both processes and the inequalities are maintained.\\
\end{itemize} 
\item There is no more type 2 particles in the TTBP, and we have to choose a type 1 particle. In this case, the reasoning is the same as above, observing moreover that, necessarily, $N^{TT,1}_t\ge N^{FM,1}_t+1$.\\
\end{enumerate}

Thus, in any case, we have $N_{t+1}^{TT,1}\ge N_{t+1}^{FM,1}$ and $N_{t+1}^{TT,1}+N_{t+1}^{TT,2}\ge N_{t+1}^{FM,1}+N_{t+1}^{FM,2}$, establishing the recursion. 
\end{enumerate}
The proof of the lower bound of Theorem \ref{theorem:bounds} is concluded. 
\end{proof}
 
\section{Proof of the upper bound}\label{sec:upper}

We will proceed in three steps. In a first step of the proof, we explain the comparison of the frog model with its oriented version. This part of the proof is common to \cite{lebensztayn/machado/popov/2005},  \cite{lebensztayn2019new} and \cite{gallo2018frog}. In the second  step, we state Theorem \ref{theo:graal}, a very nice result which identify the critical parameter of the oriented version with the root of a power series. Notice that, although this result was already proved in \cite{gallo2018frog} (see Theorem 1 therein), will partially prove it in order to give support to our discussion in Section \ref{sec:discussion}. Based on this theorem, the two last steps are dedicated to obtain bounds for the critical parameter of the oriented version, and these steps only rely on calculus. This is done by first finding finite degree polynomials (of degree 6 and degree 5) which approximate the power series from below and from above,  and next, finding root approximations for these polynomials. 

Let us now give some notation. We  denote by $d(v,v')$ the distance between any two vertices of $v,v'$ of $\mathbb T_d$, defined as the number of edges of the unique path connecting them.  We write $v\le v'$ if $v$ is one of the vertices on the path between the root $o$ and $v'$, and we naturally write $v<v'$ if $v\le v'$ and $v\ne v'$. 
 
\subsection{Step one: oriented version of the frog model}

Consider the following modification of the frog model, that we call \emph{oriented frog model}: when a frog, initially at vertex $v\in\mathbb T_d$ is activated and makes its random walk, it only activates the frogs of vertices $v'$ such that $v'>v$. 

It is obvious that, if this oriented model survives, then the original model survives as well. So if we find an upper bound for the critical parameter $\hat p_c$ of the oriented model, it will also be an upper bound for the critical parameter of the original model. 

Our objective in the remaining of the proof  will be to find tight upper and lower bounds for $\hat p_c$. Notice that even though the lower bound on $\hat p_c$ is not necessary to get our upper bound for $p_c$, we will partially provide it to show that the upper bound of $\hat p_c$ is already very accurate as an estimate of $\hat p_c$ (see the discussion in Section \ref{sec:discussion}). 

\subsection{Step two: $\hat p_c$ as the root of a power series}

For any vertices $v',v$, we denote by $\{v\rightarrow v'\}$ the event that the frog at $v$, if it were active, would visit  $v'$ during his random walk. 
Lemma 2.1 of \cite{lebensztayn/machado/popov/2005} states that for any $v,v'$ such that $d(v,v')=n\ge1$, 
\begin{equation}\label{eq:geo_oriented}
\mathbb P_p(v{\rightarrow}v')=r^n
\end{equation}
where 
\begin{equation}\label{eq:r_e_p}
r=r(p,d):=\dfrac{d+1-\sqrt{(d+1)^2-4dp^2}}{2dp},
\end{equation}
a fact which will have importance later on. Notice in particular that, writing $r_c=r(\hat p_c,d)$, we can now focus on $r_c$ directly since $r$ is a continuous bijection on $[0,1]$.

The objective of the present step is to prove the following very nice theorem which gives the critical parameter as root of a power series. It is proved, indirectly, in \cite{gallo2018frog}, so we include its proof here for the sake of completeness. 

\begin{theorem}\label{theo:graal}
$\sum_{k\geq 1} d^{k}r(\hat p_c,d)^{k}\prod_{i=1}^{k-1} (1-r(\hat p_c,d)^{i})=1.$
\end{theorem}

\begin{proof}
We denote by $\{v\stackrel{c}{\rightarrow}v'\}$ the event that either $\{v\rightarrow v'\}$, or there exist $k\ge1$ and a sequence of vertices $v_1,\ldots,v_k$ such that $v=:v_0<v_1<\ldots<v_k<v_{k+1}:=v'$ such that $\cap_{i=0}^{k}\{v_{i}\rightarrow v_{i+1}\}$. In words, $\{v\stackrel{c}{\rightarrow}v'\}$ means that $v$ has started a chain of activation of frogs which in particular activates $v'$. 
By symmetry, we can use  $u_n=u_n(r,d):=\mathbb P_p(o\stackrel{c}{\rightarrow}v)$ for any $v$ such that $d(o,v)=n$. 

On the one hand we have
\[
\mathbb P_p(\cup_{v:d(o,v)=n}\{o\stackrel{c}{\rightarrow}v\})\searrow\mathbb P_p(\text{survival of the oriented model}).
\]
Thus
\[
\mathbb P_p(\cup_{v:d(o,v)=n}\{o\stackrel{c}{\rightarrow}v\})\le d^n\mathbb P_p(o\stackrel{c}{\rightarrow}v)=:d^nu_n.
\]
In other words, 
\begin{equation}\label{eq:lower_implication}
d^nu_n\rightarrow0\Rightarrow \mathbb P_p(\text{survival of the oriented model})=0.
\end{equation}

On the other hand, the expected number of vertices at distance $n$ of the root which have been visited by frogs is, also by symmetry
\[
\sum_{v:d(o,v)=n}\mathbb P_p(o\stackrel{c}{\rightarrow}v)=d^nu_n.
\]
Fix some $N\ge1$ and consider the following process, obtained  from the oriented model \emph{via} the following recursive procedure: 
\begin{itemize}
\item At each vertex of level $N$ which has been visited (at time $N$, since the process is oriented), we keep only one activated frog. 
\item Each frog activated at level $iN,i\ge1$ is started, and at each vertex of level $(i+1)N$ which has been visited, we keep only one activated frog. 
\end{itemize}
Observe first that this modified process is dominated by the oriented process, since at each step $i\ge1$ of the recursive procedure we only keep 1 activated frog at the visited vertices at distance $iN$ of the root. Moreover, the frog we keep at $iN$, given it has reached this level, can be substituted by a new active frog, because of the loss of memory property of the geometric distribution \eqref{eq:geo_oriented}. It follows that the number of visited vertices at level $iN,i\ge1$ of this modified process has the same distribution as the number of individuals at the $i^{\text{th}}$ generation in a Galton-Watson tree with expected offspring $d^Nu_N$. If $d^Nu_N>1$,  the Galton-Watson has positive probability to survive,  and thus so does the oriented model. In other words
\begin{equation}\label{eq:implication_upper}
\exists N\ge1:d^Nu_N>1\Rightarrow  \mathbb P_p(\text{survival of the oriented model})>0.
\end{equation}

To continue, we need to study $u_n=\mathbb P_p(o\stackrel{c}{\rightarrow}v)$ and its limiting properties, and it is precisely in that study   that \cite{gallo2018frog} differs from \cite{lebensztayn/machado/popov/2005} and \cite{lebensztayn2019new} (see Section \ref{sec:discussion} for a discussion). Indeed, \cite{gallo2018frog} used a result of \cite{gallo/garcia/junior/rodriguez/2014} implying that $u_n$ is the probability that an undelayed renewal sequence, with inter-renewal distribution $f_k=r^{k}\prod_{i=1}^{k-1} (1-r^{i}),k\ge1$, has a renewal at time $n$. This allows them to conclude, using renewal theory  (see Section 4 in \cite{gallo2018frog}) that  $u_\infty(r,d):=\lim u_n^{1/n}$ exists, is  continuous in $r$, and it satisfies the equality
\begin{equation}\label{eq:graalinho2}
\sum_{k\geq 1} r^{k}u_\infty^{-k}\prod_{i=1}^{k-1} (1-r^{i})=1.
\end{equation}

Together with \eqref{eq:lower_implication} and \eqref{eq:implication_upper}, the existence of $u_\infty(r,d)$ gives us
\begin{align*}
u_\infty>1/d&\Rightarrow\mathbb P_p(\text{survival of the oriented model})>0\\
u_\infty<1/d&\Rightarrow\mathbb P_p(\text{survival of the oriented model})=0.
\end{align*}
The continuity of $u_\infty(r,d)$ allows us conclude that $r_c$ satisfies 
\begin{equation}\label{eq:graalinho}
u_\infty(r_c,d)=\frac{1}{d}.
\end{equation}
Putting \eqref{eq:graalinho} together with  \eqref{eq:graalinho2} concludes the proof of the theorem.
\end{proof}

\subsection{Third step: bounds for $\hat p_c$ as zeros of polynomials}

Using   \eqref{eq:r_e_p} and and Theorem \ref{theo:graal}, the way to proceed now is to find $\underline r=\underline r(d)$ and $\overline r=\overline r(d)$ such that 
\begin{equation}\label{eq:rs}
\underline r\le r_c\le \overline r
\end{equation}
which yields
\begin{equation}\label{eq:r_vs_p}
\frac{(d+1)\underline r}{1+d\underline r^2}\le \hat p_c\le \frac{(d+1)\overline r}{1+d\overline r^2}.
\end{equation}

A first characterization that we will give of $\underline r$ and $\overline r$ is as zeros of polynomials.

\begin{proposition}\label{corollary:lower and upper bounds as root of a polynomial}
Let  $\underline r$ (\emph{resp.} $\overline r$), denote the unique root of the polynomial $L(r):=-d^3 r^6+d^3 r^5+d^2 r^3-2 d r+1$ (\emph{resp.} of the polynomial $U(r):=d^3 r^5+d^2 r^3-2 d r+1$) in $r\in(0,1/d)$. Then 
\[
\underline r\le r_c\le \overline r.
\]
\end{proposition}

\begin{proof}
We have the inequalities for any $k\ge3$
\begin{equation}\label{eq:inequalities_rec}
1-r-r^2<(1-r-r^2+r^k\le) \prod _{i=1}^{k-1} \left(1-r^i\right) \leq (1-r)(1-r^2)
\end{equation}
where the lower bound between parenthesis follows by recursion and the two others are trivial. Inequalities \eqref{eq:inequalities_rec}  imply that
\[
f_{\text{inf}}(r)\le\sum_{k\geq 1} d^{k}r^{k}\prod_{i=1}^{k-1} (1-r^{i})\le f_{\text{sup}}(r)
\]
with 
\[
f_{\text{inf}}(r)=dr+d^2(1-r)r^2+(1-r-r^2)\left(-\frac{d^3 r^3}{d r-1} \right)
\]
and
\[
f_{\text{sup}}(r)=dr+d^2(1-r)r^2+(1-r)(1-r^2) \left(-\frac{d^3 r^3}{d r-1} \right ).
\]
With Theorem \ref{theo:graal} we conclude that 
\[
f_{\text{inf}}(r_c)\le 1\le f_{\text{sup}}(r_c).
\]
With some further algebra, we notice that
\begin{align*}
f_{\text{inf}}(r_c)\le 1&\Leftrightarrow d^3 r_c^5+d^2 r_c^3-2 d r_c+1\ge0\\
f_{\text{sup}}(r_c)\ge1&\Leftrightarrow -d^3 r_c^6+d^3 r_c^5+d^2 r_c^3-2 d r+1\le 0.
\end{align*} 
But since both polynomials are decreasing on $(0,1/d)$, then we conclude that the root $\underline r$ of $L(r)=-d^3 r^6+d^3 r^5+d^2 r^3-2 d r+1$ and the root $\overline r$ of $U(r)=d^3 r^5+d^2 r^3-2 d r+1$ will satisfy
\[
\underline r\le r_c\le \overline r
\]
as stated by the lemma.
\end{proof}
  
\subsection{Fourth (last) step: explicit bounds for $\hat p_c$ as functions of $d$}\label{sec:lastep}
In principle, we could simply seek for the exact expression of the zeros stated in Proposition \ref{corollary:lower and upper bounds as root of a polynomial}. However, this leads to very complicated expressions. So we will do one further step to get approximations of $\overline r$ and $\underline r$.

\begin{lemma}\label{corollary: explicit formula for critical parameter for directed tree}
For $d\geq2$
\begin{equation*}
\underline r(d)\ge r_L:= \frac{5 - 8 d - 16 d^2 + 64 d^3}{12 d - 20 d^2 - 48 d^3 + 128 d^4}
\end{equation*}
and
\begin{equation*}
\overline r(d)\leq r_U:= \frac{2-\frac{1}{14d^2}-4 d}{5 d-8 d^2}.
\end{equation*}
\end{lemma}

Putting Display \eqref{eq:r_vs_p}, Proposition \ref{corollary:lower and upper bounds as root of a polynomial} and Lemma \ref{corollary: explicit formula for critical parameter for directed tree} together, we get explicit lower and upper bounds for $\hat p_c$ as function of $d$. In conjunction with Step 1, this in particular proves the upper bound given in Theorem \ref{theorem:bounds}.

\section{Discussion on the  bounds and their proofs}\label{sec:discussion}

\subsection{Further literature on related models}

Under the name ``frog model'', a vast literature as been developed in recent years. Here we highlight two other papers focussing on trees. The first one,  \cite{hermon2018frogs}, considers the speed of the spread and the final proportion of activated frogs, on finite trees and with frogs having a.s. finite lifetime (not necessarily starting with one sleeping frog \emph{per} site). Although the situation and problem are slightly different than us here, the paper also gives a nice account of the recent literature for general frog models. The second one, much more related to our, is \cite{hoffman2017recurrence}, which proves transience and recurrence of the frog model on infinite trees when the frogs have infinite lifetime (and starting with one sleeping frog \emph{per} site). In this paper, the authors use a similar argument as the one used here to get the upper bound, designing a multi-type branching process which dominates the frog model of interest. We will come back to this in Section \ref{sec:lb} when we will discuss lower bounds. 

\subsection{The upper bounds and their proofs}

Our objective here is to compare the upper bounds of the literature, as well as their proofs. Specifically, we have to compare the works of \cite{lebensztayn/machado/popov/2005}, \cite{gallo2018frog} and \cite{lebensztayn2019new}, and the present work.

As already mentioned, the proofs of these works have Step 1 in common: they consider the oriented version of the frog model. The main difference is in Step 2. 
\cite{lebensztayn/machado/popov/2005} used the fact that, finding a solution, in $p$, for $d^ku_k(p,d)=1$ inside the interval $(0,1/d)$, yields an upper bound for $p_c$. This is the content of Theorem 3.1 therein, and it is a fact which can also be concluded from \eqref{eq:implication_upper}  above. The problem is that this yields a bound which depends on $k$, and as they observe, it is not obvious whether this sequence of upper bounds is decreasing, so they cannot make the limiting procedure at this step. Instead, they consider a sequence $v_k\le u_k$, and work on the asymptotic of the sequence of solutions of $d^kv_k$ as function of $r$, or, equivalently, of $p$ (see Lemmas 4.1, 4.2 and 4.3 therein). It is interesting how this approach differs from ours, on a mathematical ground. Formally, what \cite{lebensztayn/machado/popov/2005} obtained is that (see Display (4.2) therein)
\begin{equation}\label{eq:popov_at_n}
u_k=r^k\prod_{i=1}^{k-1} (1-r^{i})+\sum_{j=1}^{k-1}r^{k-j}u_j\prod_{l=1}^{k-j-1} (1-r^{l}).
\end{equation}
At this point what they decided to do, instead of studying the asymptotic behavior in $k$, is to take 
\[
v_k=r^k\prod_{i=1}^{k-1} (1-r)+\sum_{i=1}^{k-1}r^{k-j}u_j\prod_{l=1}^{k-j-1} (1-r)=r^k(1-r)^{k-1}+\sum_{i=1}^{k-1}r^{k-j}u_j(1-r)^{k-j-1}
\]
which clearly satisfies $v_k\le u_k$. Observe that this amounts simply to substitute $r^i$ and $r^l$ by $r$ into   the products of \eqref{eq:popov_at_n}. They can then work asymptotically with the solutions of  $d^kv_k=1$ and conclude  their bound $r_c\le 1-\sqrt{{(d-1)}/{d}}$, which is actually the solution of $dr(2-r)=1$ for $r\in(0,1/d)$. 
It was also remarked by \cite{lebensztayn/machado/popov/2005} that substituting the $r^i$ and $r^l$ by $r^2$ for any $k\ge2$ would naturally yield tighter, yet more complicated bounds for $r_c$. Indeed, they state that $r_c$ would be the root $\bar r_U$ of $\bar U(r):=dr^4-d(d+1)r^3+2dr-1$ in $(0,1/d)$. It is precisely what  \cite{lebensztayn2019new} used, yielding yet another refinement of the bound, although very complicated (see Definition 2.1 and Theorem 2.2 therein).

On the contrary, what we do (and what was done by \cite{gallo2018frog}) in Step 2 is that we directly work asymptotically on \eqref{eq:popov_at_n} using renewal theory, and this yields Theorem  \ref{theo:graal}. In particular, notice that the bounds of \cite{lebensztayn/machado/popov/2005} and \cite{lebensztayn2019new} are direct consequences of Theorem \ref{theo:graal} as well: indeed, substituting, in the products, $r^i$ and $r^l$ by $r$, yields the polynomial $dr(2-r)=1$ used by \cite{lebensztayn/machado/popov/2005}, and substituting by $r^i$ and $r^l$ by $r^2$ (for $i,l\ge2$) yields the polynomial $\bar U(r):=dr^4-d(d+1)r^3+2dr-1$ used by \cite{lebensztayn2019new}. 

Clearly our upper bound \eqref{eq:bounds} is simpler than the one of \cite{lebensztayn2019new}. In order to see that it  is also tighter, it is enough to notice (we did this using Mathematica) that $\bar U(r_U)<0$ and that $\bar U$ is monotonically increasing on $(0,1/d)$ (see Lemma \ref{corollary: explicit formula for critical parameter for directed tree} for the definition of $r_U$). 

\subsection{The lower bounds and their proofs}\label{sec:lb}

The original lower bound of \cite{alves2002-2} was based on a simple coupling with a one-type branching process in which  each particle could have 0 offsprings with probability $1-p$, 1 offspring with probability $p/(d+1)$ (for frogs coming back) and 2 offsprings with probability $pd/(d+1)$. In other words, their coupling took into account the fact that, except for the frog initially at the root, an active frog necessarily is surrounded by at least one neighbouring site. What we noticed in the present paper  is that after a certain number of steps (almost-surely finite), any activated frog which is not at the tip of the visited cluster has at least two visited neighbouring sites. In order to take this into account,  we needed to consider a two-type branching process. 

The frog model can be seen as an ``infinite types'' branching process, and it is naturally possible to improve further our reasoning. The idea would be to compare the frog model with  branching processes having more and more types.  For instance, a simple modification of the multi-type branching process used by \cite{hoffman2017recurrence} in the proof of their Proposition 19 could be used as well in our setting. This amounts essentially in adding types corresponding to keeping track of two simultaneously activated particles. Doing so, we are able, for instance, to prove that $p_c(2)<p_c(3)$, but it is impossible to get explicit expression holding for any $d$ based on so much types. More generally, we tried such method up to a certain level, and obtained slightly tighter lower bounds with very involved expressions, and we preferred to keep it simple at the cost of precision. 

\subsection{Concluding remark}

Tighter bounds on $\hat p_c$ can be obtained from Theorem \ref{theo:graal}, this is a matter of root approximation for the power series of Theorem \ref{theo:graal}. However, as pointed by Table \ref{tab:Comparison between Theorems our theorem and Utria theorem}, such improvements would be almost insignificant compared to the distance to the lower bounds of $p_c$ since we are limited by the lower bound on $\hat p_c$, already very close to the upper bound. For this reason we mentioned in Introduction that further works should either improve the lower bound or find another approach (different from the comparison with the oriented version of the frog model) to get upper bounds for $p_c$. 
%Owing to what we said in Subsection \ref{sec:lb}, our bet is that it is the upper bound which needs to be improved through a new method of proof.  

\begin{table}[htbp]
  \centering
  \caption{Lower bound (LB) and upper bound (UB) of $p_c$ from Theorem \ref{theorem:bounds} and lower bound on $\hat p_c$ from Lemma \ref{corollary: explicit formula for critical parameter for directed tree}.}
    \begin{tabular}{cccc}
     \textbf{d} & \textbf{LB on $p_c$} &\textbf{LB on $\hat p_c$} & \textbf{UB on $\hat p_c$ and $p_c$} \\
    2     &0.6261364 &0.7103674 & 0.7137989 \\
    3     &0.5835921& 0.6419859 & 0.6428580 \\
    4     &0.5625890&0.6071563 & 0.6074957 \\
    5     &0.5500385&0.5860557 & 0.5862210 \\
    6     &0.5416859 &0.5719015 & 0.5719940 \\
    7     &0.5357250 &0.5617475 & 0.5618043 \\
    8     &0.5312564 &0.5541074 & 0.5541448 \\
    9     &0.5277818 &0.5481503 & 0.5481761 \\
    10   &0.5250027& 0.5433751 & 0.5433937 \\
    20   &0.5125001 &0.5217793 & 0.5217815 \\
    50   &0.5050000 &0.5087345 & 0.5087346 \\
    100 &0.5025000 &0.5043711 & 0.5043711 \\
    \end{tabular}%
  \label{tab:Comparison between Theorems our theorem and Utria theorem}%
\end{table}

\section{Appendix: Root approximations}
 
 As promised, we now prove Lemma \ref{corollary: explicit formula for critical parameter for directed tree}.
 \begin{proof}
For the lower bound, we use the Newton-Raphson method for root approximation of $L(r)=-r^3 r^6+d^3 r^5+d^2 r^3-2 d r+1$. The second derivative of $L(r)$ is positive for all $r\in(0,1/d)$, so  $L(r)$ is convex on this interval. Thus, the approximation calculated from the Newton-Raphson method will be smaller than the root $\underline r$. Recall that the Newton-Raphson iterative method is started from some value $t_0$ and for any $n\ge1$
\begin{eqnarray*}
t_{n} &=& t_{n-1} - \frac{L(t_{n-1})}{L'(t_{n-1})}.
\end{eqnarray*}
Starting with $t_0=0$ and iterating two time, we get after some algebraic manipulations
\begin{equation*}
t_2=\frac{5 - 8 d - 16 d^2 + 64 d^3}{12 d - 20 d^2 - 48 d^3 + 128 d^4}.
\end{equation*}
This concludes the proof of the lower bound. 

For the upper bound, the proof of the Newton-Raphson method is much longer so we will proceed simply proving that $r_U\ge \bar r$, which is faster. We proceed in two steps: we first show that the polynomial $U$ is monotonically decreasing in $r$ around the root of interest, and secondly, that $U(r_U)<0$ for any $d\ge2$. 

For the first step, notice that
\[
U'(r)<0\Leftrightarrow 5d^2 r^4+3d r^2-2 <0,
\]
This is equivalent to $r<\sqrt{{1}/{(5d)}}$. It only remains to get sure that $r_U<\sqrt{{1}/{(5d)}}$. To see this, we come back to the proof of Proposition \ref{corollary:lower and upper bounds as root of a polynomial}, and notice that 
\[
(1-r)^{k-1}\le \prod _{i=1}^{k-1} \left(1-r^i\right)
\] 
leads to an intermediary polynomial $f_{\text{inf}}(r)\le dr(2-r)\le \sum_{k\geq 1} d^{k}r^{k}\prod_{i=1}^{k-1} (1-r^{i})$. Thus $dr_U(2-r_U)\leq1$, meaning that indeed $r_U\le 1-\sqrt{{(d-1)}/{d}}$ which is strictly smaller than $\sqrt{{1}/{(5d)}}$ for any $d\ge2$. 

To prove  the second step, we used Mathematica to write $U(r_U(d))<0$ as
\[
-\frac{1}{537824 \,d^{12} (8 d - 5)^5}\bar U(d)<0
\]
where 
\begin{align*}
\bar U(d)&=211441664 \,d^{14} - 801511424 \,d^{13} + 988904672 \,d^{12} - 496642048\, d^{11}  \\&+ 72342816 \,d^{10} + 14993216 \,d^{9} - 2579360 \,d^{8} - 918064 \,d^{7} \\&+ 203840 \,d^{6} + 26460 \,d^{5} - 7840 \,d^4 - 280 \,d^3 + 140 \,d^2 - 1.
\end{align*}
We are done if we prove that  $\bar U(d)>0$ for any $d\ge2$. Using Cauchy bounds, we know that the largest root of this polynomial is smaller or equal to (for $i=0,\ldots,14$ we write $a_i$ for the coefficient of order $i$)
\[
m:=1+\max\left\{\left|\frac{a_{13}}{a_{14}}\right|,\ldots,\left|\frac{a_0}{a_{14}}\right|\right\}=1+\frac{988904672}{211441664}<6. 
\]
The proof is concluded noticing that for $d=2,\ldots,6$, $\bar U(d)>0$.

\end{proof}

\bibliographystyle{plain}
\bibliography{sandrobibli}
\end{document}